\documentclass[a4paper,12pt]{article}
\usepackage{amssymb,amsmath,amsthm,amstext,graphics,amsfonts,hyperref}
\usepackage{cite}

\DeclareMathOperator{\ord}{Ord}

\newtheorem{theorem}{Theorem}[]
\newtheorem{lemma}[theorem]{Lemma}
\newtheorem{corollary}[theorem]{Corollary}
\newtheorem{proposition}[theorem]{Proposition}
\theoremstyle{definition}

\newtheorem{definition}[theorem]{Definition}
\newtheorem{alg}[theorem]{Algorithm}

\font\msbm=msbm10 at 12pt

\newcommand{\F}{\mbox{\msbm F}}

\def\x{\mathbf{x}}

\def\y{\mathbf{y}}

\def\Xixi{\Xi_{S,S_1,S_2,t}}
\def\xixi{\xi_{S,S_1,S_2,t}}
\def\teta{\Sigma_{S,k}\circ\Gamma_{S_1,S_2,t}}
\def\tilam{\tilde{\lambda}_{S_1,S_2}}
\title{Skew-Cyclic Codes over $B_k$}
\author{Irwansyah\\
{\it \small Mathematics Department,}\\
{\it \small Universitas Mataram, Jl. Majapahit 62, Mataram,}\\
{\it INDONESIA}\\
\vspace{0.1cm}\\
Aleams Barra, Intan Muchtadi-Alamsyah, Ahmad Muchlis,\\
{\it \small Algebra Research Group,}\\
{\it \small Faculty of Mathematics and Natural Sciences,}\\
{\it \small Institut Teknologi Bandung, Jl. Ganesha 10, Bandung, 40132,}\\
{\it \small INDONESIA}\\
\vspace{0.1cm}\\
Djoko Suprijanto \footnote{Contact: djoko@math.itb.ac.id}\\
{\it \small Combinatorial Mathematics Research Group,}\\
{\it \small Faculty of Mathematics and Natural Sciences,}\\
{\it \small Institut Teknologi Bandung,
Jl. Ganesha 10, Bandung, 40132,}\\
{\it \small INDONESIA}}
\date{}

\begin{document}
\maketitle
\begin{abstract}
 In this paper we study the structure of $\theta$-cyclic codes over the ring $B_k$
 including its connection to quasi-$\tilde{\theta}$-cyclic codes over
finite field $\mathbb{F}_{p^r}$ and
skew polynomial rings over $B_k.$ We also characterize Euclidean
self-dual $\theta$-cyclic codes over the rings. Finally, we give the generator polynomial for such codes and some examples of
optimal Euclidean $\theta$-cyclic codes.
%, including
%extended $[36,18,11]$ code which based on \cite{boucher}, improve previous known bound for self-dual codes. \\
\\[0.25cm]

\end{abstract}

\section{Introduction}
Cyclic codes are some of the most interesting families of codes because of their rich algebraic structures.
In 2007, Boucher {\it et. al.} \cite{boucher} introduced a notion of skew-cyclic codes or $\theta$-cyclic codes
as a generalization of cyclic codes.
This class of codes has been studied over certain finite rings such as
finite fields \cite{boucher}, the ring $\mathbb{F}_2+v\mathbb{F}_2,$ where $v^2=v$ \cite{abualrub},
and very recently over the rings  $A_k$ \cite{irw}, and produced several
optimal codes including optimal Euclidean self-dual code $[36,18,11]$ over
$\mathbb{F}_4$ which improve previous known bound
for self-dual
codes \cite{boucher}. Moreover, the codes have a close connection with  modules over skew-polynomial rings as shown in
\cite{abualrub},\cite{boucher},\cite{bou-geis},\cite{gao}, and \cite{irw}.

%On the other hand, codes over the ring $A_k=\mathbb{F}_2[v_1,v_2,\dots,v_k]/\langle v_i^2v_i,v_iv_j-v_jv_i\rangle$
%for all $i,j=1,\dots,k,$ as shown
%in \cite{dougherty-ceng}, have interesting properties as well as recursive Gray map,
%minimum generating sets,
%MacWilliams relations for complete and Lee weight enumerator, self dual codes and its relation to lattices.

In this paper, we define $B_k,$ an algebra over a finite field $\mathbb{F}_{p^r},$
which is a natural generalization of the ring $A_k,$ and
we study its structures including the shape of its maximal ideals and automorphisms.
We also study the structures of $\theta$-cyclic codes over the ring $B_k.$
We focus on its connection to quasi-$\theta$-cyclic codes over
finite fields and skew polynomial rings over $B_k.$ We give the generators for such codes.
Moreover, we also study self-dual $\theta$-cyclic codes over
such rings and give some examples of optimal codes.

\section{The ring $B_k:$ basic facts}
In this section we provide some basic facts regarding the ring $B_k.$  Some properties are easy to derived, but
we include here for the reader's convenience.

Let $\mathbb{F}_{p^r}$ be the field extension of degree $r$ of prime field $\mathbb{F}_p,$ for some positive integer $r.$
The ring $B_k$ is defined as $B_k:=\mathbb{F}_{p^r}[v_1,v_2,\dots,v_k]/\langle v_i^2-v_i,v_iv_j-v_jv_i \rangle$
for all $i,j=1,\dots,k.$ For example, when $k=1,$ then $B_1=\mathbb{F}_{p^r}+v\mathbb{F}_{p^r},$
where $v^2=v.$ For convenience, we let $B_0=\mathbb{F}_{p^r}.$ The ring $B_k$
can be considered as a commutative algebra over $\mathbb{F}_{p^r}.$ Let $\mathcal{H}$
be the collection of all subsets of $\{1,\dots,k\}.$ Then, we have the following observation.

\begin{lemma}
The ring $B_k$ can be viewed as an $\mathbb{F}_{p^r}$-vector space with dimension $2^k$ whose basis consists of elements of the form
$\prod_{i\in H}w_i,$ where $H\in\mathcal{H}$ and $w_i\in\{v_i,1-v_i\}$ for $1\leq i\leq k.$
\label{lemavector}
\end{lemma}
\begin{proof}

As we can see, every element $a\in B_k$ can be written as $a=\sum_{H\in \mathcal{H}}\alpha_Hv_H,$
for some $\alpha_H\in\mathbb{F}_{p^r},$
where  $v_H=\prod_{i\in H} v_i$ and $v_\emptyset=1.$
Therefore, $B_k$ is a vector space over $\mathbb{F}_{p^r}$ with basis consists
of elements of the form $v_H=\prod_{i\in H} v_i,$ where $v_\emptyset=1$ and
there are $\sum_{j=0}^k{k\choose j}=2^k$ elements of basis. Now, we will show
that the set $\{1, w_{H_2},\dots, w_{H_{2^k}}\}$ is also a basis,
where $H_2,\dots,H_{2^k}\in\mathcal{H}$ such that $|H_i|\leq |H_j|$ for all $2\leq i<j\leq 2^k.$
Consider,
\[
\alpha_1 + \alpha_2 w_{H_2}+\dots+\alpha_{2^k} w_{H_{2^k}}=0
\]
for some $\alpha_i\in\mathbb{F}_{p^r},$ for all $i=1,\dots,2^k,$ which gives,
\[
-\alpha_1=\alpha_2 w_{H_2}+\dots+\alpha_{2^k} w_{H_{2^k}}.
\]
If $\alpha_1\not= 0,$ then $\xi_1=\left(\alpha_2 w_{H_2}+\dots+\alpha_{2^k} w_{H_{2^k}}\right)$ is a unit, a contradiction to
$\xi_1\in\langle w_1,\dots,w_k\rangle.$ So, $\alpha_1=0,$ which means,
\[
-\left(\alpha_2 w_{H_2}+\dots+\alpha_{k+1}w_{H_{k+1}}\right)=\alpha_{k+2}w_{H_{k+2}}+\dots+\alpha_{2^k} w_{H_{2^k}}.
\]
If $\left(\alpha_2 w_{H_2}+\dots+\alpha_{k+1}w_{H_{k+1}}\right)\not=0,$
then it is a contradiction to the fact that $|H_j|\geq 2,$ for all
$j=k+2,\dots, 2^k.$ Consequently, $\left(\alpha_2 w_{H_2}+\dots+\alpha_{k+1}w_{H_{k+1}}\right)=0.$
We have to note that, the set
$\mathcal{B}_1=\{1,w_{\{1\}},w_{\{1,2\}},w_{\{2\}},\dots,w_{\{1,2,\dots,k\}}\}$
is also linearly independent over $\mathbb{F}_{p^r},$
because $B_k$ is a vector space over $\mathbb{F}_{p^r}$ with elements of basis are of the
form $v_H,$ where $H\in\mathcal{H}.$ Therefore, $\left(\alpha_2 w_{H_2}+\dots+\alpha_{k+1}w_{H_{k+1}}\right)=0$ gives
$\alpha_2=\dots=\alpha_{k+1}=0.$
On continuing this process, we have $\alpha_1=\dots=\alpha_{2^k}=0,$ which means they are linearly independent
over $\mathbb{F}_{p^r}.$
\end{proof}

The following result is an immediate consequence of the above lemma.

\begin{lemma}
The ring $B_k$ has characteristic $p$ and cardinality $(p^r)^{2^k}.$
\end{lemma}
\begin{proof}
It is immediate since the characteristic of $\mathbb{F}_{p^r}$ is $p,$ and $B_k$ can be viewed as a $\mathbb{F}_{p^r}$-vector space
with dimension $\sum_{i=0}^k{k\choose i}=2^k.$
\end{proof}

The following lemma gives a characterization for zero divisor elements in $B_k.$

\begin{lemma}
An element $\omega\in B_k$ is a zero divisor if and only if $\omega\in \langle w_1,w_2,\dots,w_k\rangle,$ where $w_i\in\{v_i,1-v_i\}$ for
$1\leq i \leq k.$
\label{lemmazero}
\end{lemma}
\begin{proof}
($\Longleftarrow$) It is clear that, $v_i(1-v_i)=0,$ for all $i=1,\dots,k.$ Therefore,
if $\omega\in\langle w_1,w_2,\dots,w_k\rangle,$ then it is a
zero divisor in $B_k.$\\[0.25cm]
($\Longrightarrow$) Consider the equation,
\[(\alpha+\beta v_k)(\gamma+\epsilon v_k)=a+bv_k\]
given $\alpha+\beta v_k,a+bv_k\in B_k,$ for some $\alpha,\beta, a,b\in B_{k-1}.$
We have $\gamma=a \alpha^{-1}$ and $\epsilon=(b-\beta a) (\alpha(\beta+\alpha))^{-1}.$ Therefore, if $a+bv_k=1,$ then $\gamma=1$ and
$\epsilon=-\beta (\alpha(\beta+\alpha))^{-1}.$ Which implies, $\alpha+\beta v_k$ is a unit if and only if $\alpha$ and $\alpha+\beta$ are also units.
Considering this observation for elements in $B_{k-1}, B_{k-2},\dots, B_1,$ we have $\alpha+\beta v\in B_1$ is a unit if and only if
$\alpha,\alpha+\beta\in \mathbb{F}_{p^r}$ are non zero elements.
Since, every element in finite commutative ring is either a unit or a zero divisor, we can see that the only zero divisors in $B_1$ are the elements in
the ideals generated by $\beta v$ or $\alpha(1-v).$ By generalizing this result recursively, we have the intended conclusion.
\end{proof}

Also, we can easily show that $I=\langle w_1,w_2,\dots,w_k\rangle$ is a maximal ideal in $B_k.$

\begin{lemma}
Let $I=\langle w_1,w_2,\dots,w_k\rangle,$ where $w_i\in\{v_i,1-v_i\}$ for $1\leq i \leq k.$
Then $I$ is a maximal ideal in $B_k.$
\label{maxideal}
\end{lemma}
\begin{proof}
Consider the quotient ring $B_k/I.$ If $v_i\in I,$ then $1-v_i\equiv 1\mod I,$ and if $1-v_i\in I,$ then $v_i=1-(1-v_i)\equiv 1\mod I.$
Consequently,
$B_k/I$ is a field. So, $I$ is a maximal ideal. Moreover, by Lemma~\ref{lemavector}, every element $a$ in $B_k,$ can be written as
$a=\sum_{H\in\mathcal{H}}\alpha_Hw_H,$ for some $\alpha_H\in\mathbb{F}_{p^r}.$ Therefore, we have $B_k/I\cong\mathbb{F}_{p^r}.$
\end{proof}

\begin{lemma}
$\alpha^{p^r}=\alpha,$ for all $\alpha\in B_k.$
\label{pangkat}
\end{lemma}
\begin{proof}
Let $\alpha=\sum_{H\in \mathcal{H}}\alpha_H v_H,$ for some $\alpha_H\in\mathbb{F}_{p^r},$ where $v_H=\prod_{j\in H}v_j.$ Then, for any $H_1\in\mathcal{H},$ consider
\[
\alpha^{p^r}=\sum_{i=0}^{p^r}{p^r\choose i}\left(\alpha_{H_1}v_{H_1}\right)^{i}\left(\sum_{H\not= H_1}\alpha_Hv_H\right)^{p^r-i}=\alpha_{H_1}v_{H_1}
+\left(\sum_{H\not= H_1}\alpha_Hv_H\right)^{p^r}
\]
since $\mathbb{F}_{p^r}$ has characteristic $p$ and $\beta^{p^r-1}=1$ for all $\beta\in\mathbb{F}_{p^r}.$ If we continue this procedure, then we have $\alpha^{p^r}=\alpha.$
\end{proof}

The proposition below shows that $B_k$ is a principal ideal ring.

\begin{proposition}
Let $I=\langle \alpha_1,\dots,\alpha_m\rangle$ be an ideal in $B_k,$ for some $\alpha_1,\dots,\alpha_m\in B_k.$ Then,
\[
I=\left\langle\sum_{A\subseteq\{1,\dots,m\},A\not=\emptyset}(-1)^{|A|+1}\left(\prod_{j\in A}\alpha_j\right)^{p^r-1}\right\rangle.
\]
\label{genideal}
\end{proposition}
\begin{proof}
Consider $\alpha_i\sum_{A\subseteq\{1,\dots,m\},A\not=\emptyset}(-1)^{|A|+1}\left(\prod_{j\in A}\alpha_j\right)^{p^r-1}.$ For any $A\subseteq \{1,\dots,m\},$
if $i\in A,$ then
\[
\alpha_i(-1)^{|A|+1}\left(\prod_{j\in A}\alpha_j\right)^{p^r-1}=(-1)^{|A|+1}\alpha_i\left(\prod_{j\in A-\{i\}}\alpha_j\right)^{p^r-1}
\]
since $\alpha_i^{p^r}=\alpha_i$ by Lemma \ref{pangkat}. Consequently, there is a unique $A'=A-\{i\}\subseteq\{1,\dots,m\},$ such that
\[
\alpha_i\left((-1)^{|A|+1}\left(\prod_{j\in A}\alpha_j\right)^{p^r-1}+(-1)^{|A'|+1}\left(\prod_{j\in A}\alpha_j\right)^{p^r-1}\right)=0.
\]
Otherwise, if $i\not\in A,$ then there is a unique $A''=A\cup\{i\}\subseteq\{1,\dots,m\}$ such that
\[
\alpha_i\left((-1)^{|A|+1}\left(\prod_{j\in A}\alpha_j\right)^{p^r-1}+(-1)^{|A''|+1}\left(\prod_{j\in A}\alpha_j\right)^{p^r-1}\right)=0.
\]
So, every term will be vanish except $\alpha_i\alpha_i^{p^r}=\alpha_i.$ Therefore,
\[
I\subseteq \left\langle \sum_{A\subseteq\{1,\dots,m\},A\not=\emptyset}(-1)^{|A|+1}\left(\prod_{j\in A}\alpha_j\right)^{p^r-1}\right\rangle.
\]
It is clear that
\[
\left\langle \sum_{A\subseteq\{1,\dots,m\},A\not=\emptyset}(-1)^{|A|+1}\left(\prod_{j\in A}\alpha_j\right)^{p^r-1}\right\rangle\subseteq I.
\]
Thus, $I=\left\langle \sum_{A\subseteq\{1,\dots,m\},A\not=\emptyset}(-1)^{|A|+1}\left(\prod_{j\in A}\alpha_j\right)^{p^r-1}\right\rangle.$
\end{proof}

% \begin{lemma}
% $B_k$ is isomorphic to finite product of $\mathbb{F}_{p^r}$ via Chinese Remainder Theorem.
% \label{chinese}
% \end{lemma}
% \begin{proof}
% Let $m_i$ be the maximal as in Lemma \ref{maxideal}. By Lemma~\ref{lemavector}, every element $a$ in $B_k,$ can be written as
% $a=\sum_{H\in\mathcal{H}}\alpha_Hw_H,$ for some $\alpha_H\in\mathbb{F}_{p^r}.$ Therefore, we have $B_k/m_i\cong\mathbb{F}_{p^r}.$
% Taking all $m_i$ such that they are pairwise coprime, then we have the result.
% \end{proof}

The following proposition shows that the ideal in Lemma \ref{maxideal} is the only maximal ideal in $B_k.$

\begin{proposition}
An ideal $I$ in $B_k$ is maximal if and only if $I=\langle w_1,w_2,\dots,w_k\rangle,$ where $w_i\in\{v_i,1-v_i\}$ for $1\leq i\leq k.$
\label{formmaxideal}
\end{proposition}
\begin{proof}
$(\Longleftarrow)$ It is clear by Lemma \ref{maxideal}.\\[0.25cm]
$(\Longrightarrow)$ Let $J$ be a maximal ideal in $B_k.$ By Proposition~\ref{genideal}, $B_k$ is a principal ideal ring. Then,
let $J=\langle \omega\rangle,$ for some $\omega\in B_k.$ Note that, $\omega$ is not a unit in $B_k,$ so it is a zero divisor. By Lemma~\ref{lemmazero},
$\omega$ is an element of some $m_i=\langle w_1,w_2,\dots,w_k\rangle,$ which means $J\subseteq m_i.$ Consequently, $J=m_i,$ because $J$ is a maximal ideal.
\end{proof}

The following result gives a characterization of automorphism in $B_k.$
\begin{theorem}
Let $\theta$ be an endomorphism in $B_k.$ Then,
$\theta$ is an automorphism if and only if $\theta(v_i)=w_j,$ for every $i\in \{1,\dots,k\},$
and $\theta,$ when restricted to $\mathbb{F}_{p^r},$ is an element of $Gal(\mathbb{F}_{p^r}/\mathbb{F}_{p}).$
\label{teoauto}
\end{theorem}
\begin{proof}
($\Longrightarrow$) Let $J=\langle v_1,\dots,v_k\rangle$ and $J_{\theta}=\langle \theta(v_1),\dots,\theta(v_k)\rangle.$ Consider the map
\[
\begin{array}{llll}
 \lambda : & \displaystyle{\frac{B_k}{J}} & \rightarrow & \displaystyle{\frac{B_k}{J_{\theta}}} \\
	   & a+J & \mapsto & \theta(a)+J_{\theta} \\
\end{array} \]
We can see that the map $\lambda$ is a ring homomorphism.
For any $a,b\in B_k/J$ where $\lambda(a)=\lambda(b),$ let $a=a_1+J$ and $b=b_1+J$ for some $a_1,b_1\in B_k.$ As we can see,
$\theta(a_1-b_1)\in J_{\theta},$ so $a_1-b_1\in J.$ Consequently, $a-b=0+J,$ which means $a=b,$ in other words, $\lambda$ is a monomorphism.
Moreover, for any $a'\in B_k/J_{\theta},$ let $a'=a_2+J_{\theta}$ for some $a_2\in B_k,$ then there exists $a=\theta^{-1}(a_2)+J$ such that
$\lambda(a)=a'.$ Therefore, $\mathbb{F}_{p^r}\simeq B_k/J\simeq B_k/J_{\theta},$ which implies $J_{\theta}$ is also a maximal ideal.
By Proposition \ref{formmaxideal}, $J_{\theta}=\langle w_1,\dots,w_k\rangle,$ where $w_i\in\{v_i,1-v_i\}$ for $1\leq i\leq k.$
By Proposition \ref{genideal},
\[
\begin{aligned}
J_\theta &= \left\langle\sum_{A\subseteq\{1,\dots,k\},A\not=\emptyset}(-1)^{|A|+1}\left(\prod_{j\in A}w_j\right)^{p^r-1}\right\rangle \\
&=\left\langle\sum_{A\subseteq\{1,\dots,k\},A\not=\emptyset}(-1)^{|A|+1}\left(\prod_{j\in A}\theta(v_j)\right)^{p^r-1}\right\rangle
\end{aligned}
\]
which means,
\[
\sum_{A\subseteq\{1,\dots,k\},A\not=\emptyset}(-1)^{|A|+1}\left(\prod_{j\in A}w_j\right)^{p^r-1}
\]
and
\[
\sum_{A\subseteq\{1,\dots,k\},A\not=\emptyset}(-1)^{|A|+1}\left(\prod_{j\in A}\theta(v_j)\right)^{p^r-1}
\]
are associate. Therefore, $\theta(v_{i})=\beta w_j$ for some unit $\beta$ which satisfies $\left(\beta^{|A|}\right)^{p^r-1}=\beta,$
for all $A\not=\emptyset.$ Since $\beta^{p^r}=\beta$ and $\beta$ is a unit, we have $\beta^{p^r-1}=1.$ Therefore, $\beta$
must be equal to $1.$ Moreover, since $\theta$ is an automorphism, $\theta(v_i)\not=\theta(v_j)$ whenever $i\not=j.$
Also, since all automorphism of $\mathbb{F}_{p^r}$ are elements of $Gal(\mathbb{F}_{p^r}/\mathbb{F}_{p}),$ then it is
clear that when $\theta$ restricted to $\mathbb{F}_{p^r},$ it is really an element of
$Gal(\mathbb{F}_{p^r}/\mathbb{F}_{p}).$ \\[0.3cm]
($\Longleftarrow$) Suppose that $\theta(v_i)=w_j,$ and $\theta(v_i)\not=\theta(v_j)$ whenever $i\not=j.$
 By Lemma \ref{lemavector}, we can see that $\theta$ is also an
automorphism.
\end{proof}
% Next, we consider Gray map from $B_k$ to $B_{k-1}\times B_{k-1}.$ Define a map,
% \[\begin{array}{cccc}
% \varphi_k : & B_k & \rightarrow & B_{k-1}\times B_{k-1} \\
% 	    & \alpha+v_k\beta & \mapsto & (\alpha,\alpha-\beta)\\
% \end{array}\]
% We can easliy show that $\varphi_1$ is a bijection map. Then, we can define a map $\Phi_k : B_k \rightarrow \mathbb{F}_{p^s}^{2^k},$
% where for all $a\in B_k,$
% $\Phi_k(a)=\varphi_1(\varphi_2(\dots(\varphi_k(a))\dots)).$ We can extend $\Phi_k$ to be a bijection from $B_k^n$ to $(\mathbb{F}_{p^s}^{2^k})^n$
% Next, we prove that there exist an invariant element with respect to some automorphism in $B_k.$
% \begin{lemma}
% Let $\theta$ be an automorphism in $B_k$ as in Theorem \ref{teoauto} with $\beta=1,$ then there exists $\alpha\in B_k$ such that $\theta(\alpha)=\alpha.$
% \label{fixedmem}
% \end{lemma}
% \begin{proof}
% We can see that $\theta$ always consist of cycles of length maximum $k.$ Consider a cycle in $\theta,$ say
% $(w_{i_1}\;w_{i_2}\dots w_{i_t}).$ Let $a=\sum_{j=1}^t w_{i_j},$ then we have $\theta(a)=a.$
% \end{proof}

\section{Gray map}

As mentioned implicitly in Lemma \ref{lemavector}, every element in the ring $B_k$
can be written as $\alpha + \beta v_k,$ where $\alpha, \beta \in B_{k-1}.$
For $k \geq 2,$
let $\phi_k:B_k \longrightarrow B_{k-1}^2$ where
$$ \phi_k(\alpha + \beta v_k) = (\alpha, \alpha + \beta). $$
Then define \emph{the Gray map} $\Phi_k : B_k \longrightarrow \mathbb{F}_{p^r}^{2^k}$ where
$\Phi_1(\gamma) = \phi_1(\gamma),$
$\Phi_2(\gamma) = \phi_1(\phi_2(\gamma))$ and
$$ \Phi_k (\alpha) = \phi_1 (\phi_2(  \dots (\phi_{k-2} ( \phi_{k-1} ( \phi_k (\gamma)) \dots ).$$
It follows immediately that $\Phi_k(1) = {\bf 1 },$ the all one vector. We note that the Gray map is a bijection and it is a
generalization of the one in \cite{irw}.

Also, every element $a$ in $B_k$ can be written as
\[a=\sum_{H\in\mathcal{H}}\alpha_Hv_H\]

for some $\alpha_H\in\mathbb{F}_{p^r},$ where $v_H=\prod_{i\in H}v_i.$ Now, define a map $\Psi_k$ as follows.
\[
\begin{array}{lllll}
\Psi_k & : & B_k & \longrightarrow & \mathbb{F}_{p^r}^{2^k} \\
           &   & a=\sum_{i=1}^{2^k}\alpha_{H_i} v_{H_i} & \longmapsto &
           \left(\sum_{H\subseteq H_1}\alpha_H,\sum_{H\subseteq H_2}\alpha_H,\dots,
		    \sum_{H\subseteq H_{2^k}}\alpha_{H}\right).
\end{array}
\]
Here, $H_1,H_2, \ldots, H_{2^k} \in \mathcal{H},$ where $\mathcal{H}$ is a collection of subsets of $\{1,2,\ldots, k\}.$
We can show that this $\Psi_k$ map is also a bijection map and it is a permutation of the Gray map $\Phi_k.$
Therefore, we can choose all $H_i$ such that
$\Psi_k=\Phi_k.$ So, from now on we assume that we have already chosen all $H_i$ such that $\Psi_k=\Phi_k.$

Let $H$ be an element in $\mathcal{H}.$   We shall define  a set of automorphisms in the ring $B_{k}$ based on the set $H.$  Define
the map $\Theta _{i}$ by
%$\ \Theta _{i}(0)=0,$ $\Theta _{i}(1)=1,$ $\ $%
\begin{equation*}
\Theta _{i}(v_{i})=v_{i}+1\text{ \ \ \ \ \ \ \ \ \ }\text{and}\text{ \ \ \ \ \ \ }%
\Theta _{i}(v_{j})=v_{j}\text{ \ \ \ }\forall \text{ }j\neq i.
\end{equation*}%

For all $ S \in\mathcal{H}$ the
automorphism $\Theta _{S}$ is defined by:
\begin{equation*}
\Theta_{S}=\prod \limits_{i \in S}\Theta _{i}.
\end{equation*}
Note that $\Theta_S$ is an involution on the ring $B_k.$
Moreover, our  $\Theta_S$ here is a generalization of the map $\Theta_S$ in \cite{irw}.

Now, let $S_1,S_2$ be two elements of $\mathcal{H}$ with the same cardinality.
Let $\lambda_{S_1,S_2}$ be a one-on-one correspondence between $S_1$ and $S_2$
and $\lambda_{S_1,S_2}(i)=i$ for all $i\not\in S_1.$
For any $\alpha\in\mathbb{F}_{p^r},$ we define the map $\Lambda_{S_1,S_2,t}$  as
follows.
\begin{equation*}
\Lambda _{S_1,S_2,t}(\alpha v_{i})=\alpha^{p^t}v_{\lambda_{S_1,S_2}(i)},
\end{equation*}
for every $i\in S_1,$ where $0\leq t\leq r.$
Using two class of automorphisms above, we can describe all automorphisms in the ring $B_k$ as stated in
the following result.

\begin{lemma}
If $\theta$ is an automorphism in the ring $B_k,$ then there exist $S,S_1,S_2,$
three subsets of $\{1,\dots,k\},$ some integer $t,$ where $|S_1|=|S_2|$ and
$0\leq t\leq r,$ such that
\[\theta=\Theta_{S}\circ\Lambda_{S_1,S_2,t}.\]
\label{autodecomp}
\end{lemma}

\begin{proof}
Let $\phi$ be the Frobenius automorphism in $\mathbb{F}_{p^r}.$ If $\theta|_{\mathbb{F}_{p^r}}=\phi^{t'},$ then we have $t=t'.$ Now, define
\[S=\{j\;|\;\theta^{-1}(1-v_j)=v_i,\text{for some}\;i\},\]
\[S_1=\{l\;|\;\theta(v_l)=w_{l'},\text{where}\;l'\not=l\},\]
and
\[S_2=\{s\;|\;\exists s'\in S_1,\text{such that}\; \theta(v_{s'})=1-v_s\}.\]
Consider $\theta(v_i)$ for any $i$ in $\{1,\dots,k\}.$ We have four cases to consider as listed below.
\begin{itemize}
\item If $\theta(v_i)=v_i,$ then $i$ is not an element of $S,$ $S_1,$ and $S_2.$ Hence, $\theta(v_i)=(\Theta_{S}\circ\Lambda_{S_1,S_2,t})(v_i).$
\item If $\theta(v_i)=v_j,$ where $j\not=i,$ then $i\in S_1,$ but not in $S_2$ and $S.$ So, $\Lambda_{S_1,S_2,t}(v_i)=v_j$ and
      $(\Theta_{S}\circ\Lambda_{S_1,S_2,t})(v_i)=\Theta_S(v_j)=v_j=\theta(v_i).$
\item If $\theta(v_i)=1-v_i,$ then $i\in S_2\cap S$ but not in $S_1.$ So, $\Theta_S(v_i)=1-v_i$ and
      $(\Theta_{S}\circ\Lambda_{S_1,S_2,t})(v_i)=\Theta_S(v_i)=1-v_i=\theta(v_i).$
\item If $\theta(v_i)=1-v_j,$ where $j\not=i,$ then $i\in S_1$ and $j\in S_2\cap S.$ We have
      $(\Theta_{S}\circ\Lambda_{S_1,S_2,t})(v_i)=\Theta_S(v_j)=1-v_j=\theta(v_i).$
\end{itemize}
\end{proof}

Let $T$ be the matrix that performs the cyclic shift on a vector. Let $\sigma _{i,k}$ be the permutation
of  $ \{ 1, 2, \dots, 2^{k} \}$ defined by
\begin{equation*}
(\sigma _{i,k}) _{ \{ j2^{i}+1, \dots, (j+1)2^{i} \}  }= T^{2^{i-1}}(j2^{i}+1, \dots, (j+1)2^{i})
\end{equation*}
for all $0 \leq j \leq 2^{k-i}-1.$
Let $\Sigma _{i,k}$ be the permutation on
elements of ${\mathbb{F}}_{p^r}^{2^{k}}$ induced by $\sigma _{i,k}.$ That is,
for ${\mathbf{x}}=(x_{1},x_{2},\ldots ,x_{2^{k}})\in {\mathbb{F}}%
_{p^r}^{2^{k}} ,$
\begin{equation}
\Sigma _{i,k}({\mathbf{x}})=\left( x_{\sigma _{i,k} (1)},x_{\sigma
_{i,k}(2)},\ldots ,x_{\sigma _{i,k} (2^{k})}\right) \text{.}
\label{eq:induce1}
\end{equation}%

Related to the Gray map, we have the following results.

\begin{lemma}
Let  $k \geq 1$ and   $1 \leq i \leq k.$  For $x\in B_{k}$ we have
\begin{equation*}
\Sigma _{i,k}(\Phi _{k}(x))=\Phi _{k}(\Theta_i (x))
\end{equation*}
\label{thetas}
\end{lemma}
\begin{proof}
See the proof of \cite[Lemma 2.6]{irw}.
\end{proof}

We can extend the definition of $\Sigma_{i,k}$ to any element of $\mathcal{H}$ as follows.

\begin{definition}
For all $ S\in\mathcal{H}$ we define the
permutation $\Sigma_{S,k} $ by
\begin{equation*}
\Sigma _{S,k}=\prod\limits_{i\in S}\Sigma _{i,k}.
\end{equation*}
\end{definition}

It is clear that  for all $x\in B_{k}$ we have
\begin{equation}
\Sigma _{S,k}(\Phi _{k}(x))=\Phi _{k}(\Theta_S (x)).
\label{olfa}
\end{equation}

Given automorphism $\Lambda_{S_1,S_2,t},$ let $\tilde{\lambda}_{S_1,S_2}$ be a permutation on $H_i$ induced by $\lambda_{S_1,S_2},$
 here we assume $\lambda_{S_1,S_2}$
is a bijection map on $\{1,\dots,k\},$ where $\lambda_{S_1,S_2}(j)=j$ when
$j\not\in S_1.$ Then, for any $a=\sum_{H\in\mathcal{H}}\alpha_Hv_H,$
we have
\[
\Psi_k(\Lambda_{S_1,S_2,t}(a))
=\left(\alpha_{\emptyset}^{p^t},\sum_{H\subseteq H_{\tilde{\lambda}_{S_1,S_2}^{-1}(2)}}\alpha_{H}^{p^t},\dots,
\sum_{H\subseteq H_{\tilde{\lambda}_{S_1,S_2}^{-1}(2^k-1)}}\alpha_{H}^{p^t},\sum_{H\subseteq H_{2^k}}\alpha_H^{p^t}\right).
\]
The right hand side of the above equation induced a bijective map $\Gamma_{S_1,S_2,t}$
on $\mathbb{F}_{p^r}^{2^k}$ which simplify the equation to be
the following

\begin{equation}
\Psi_k\circ\Lambda_{S_1,S_2,t}=\Gamma_{S_1,S_2,t}\circ\Psi_k.
\label{irwan}
\end{equation}

Furthermore, related to any automorphism in the ring $B_k,$ we have the following result.

\begin{proposition}
Let $\theta$ be an automorphism in the ring $B_k.$ Then, there exist $S,S_1,S_2,$
three subsets of $\{1,\dots,k\}$ and some integer $t,$ where $|S_1|=|S_2|$ and
$0\leq t\leq r,$ such that
\[
\Psi_k\circ\theta=(\Sigma_{S,k}\circ\Gamma_{S_1,S_2,t})\circ\Psi_k.
\]
\end{proposition}

\begin{proof}
Apply Lemma \ref{autodecomp}, Lemma \ref{thetas}, and equation \eqref{irwan}.
\end{proof}

\section{Skew-cyclic codes}
In this section, we characterize the skew-cyclic codes over the ring $B_k.$  We define first the skew-cyclic
codes or $\theta$-cyclic codes as a generalization of cyclic codes.

\begin{definition}
Let $\theta$ be an automorphism in $B_k.$  $C\subseteq B_k^n$ is called \emph{$\theta$-cyclic code} of length $n$ over $B_k$ if
the following two conditions hold:
\begin{itemize}
\item[(1)] $C$ is a $B_k$-submodule of $B_k^n,$

\item[(2)] $T_{\theta}(c)=(\theta(c_{n-1}),\theta(c_{0}),\theta(c_{1}),\dots,\theta(c_{n-2}))\in C,$
for any $c=(c_0,c_1,\dots,c_{n-1})\in C.$
\end{itemize}
\end{definition}

We note that the code $C \subseteq B_k^n$ is called \emph{linear} over $B_k$ if $C$ satisfies the property $(1)$ above.

We may also generalize the above definition to the quasi-cyclic one.
\begin{definition}
Let $\theta$ be an automorphism in $\mathbb{F}_{p^r}.$
$C\subseteq \mathbb{F}_{p^r}^n$ is called \emph{quasi-$\theta$-cyclic} code of length $n$ over
$\mathbb{F}_{p^r}$ of index $l$ if the following two conditions hold:
\begin{itemize}
\item[(1)] $C$ is an $\mathbb{F}_{p^r}$-submodule of $\mathbb{F}_{p^r}^n,$

\item[(2)]  $T_{\theta}^l(c)=(\theta(c_{n-l\pmod n}),\theta(c_{n-l+1\pmod n}),\theta(c_{n-l+2\pmod n}),\dots,
\theta(c_{n-1-l\pmod n}))\in C,$ for any $c=(c_0,c_1,\dots,c_{n-1})\in C.$
\end{itemize}
(In this case we have $l$ is a divisor of $n.$)
\end{definition}

Let $S,S_1,S_2\subseteq \{1,2,\dots,k\},$ where $|S_1|=|S_2|$ and let
$\Xi_{S,S_1,S_2}=\xi_{S,S_1,S_2,t}\circ T^{2^k}$ be a bijective map
on elements of $\mathbb{F}_{p^r}^{n2^k}$ where $T$ is the cyclic shift modulo
$n2^k$ and $\xi_{S,S_1,S_2,t}$ defined for all elements
\[
\mathbf{x}=(x_1^{1},\dots, x_{2^k}^{1},x_1^{2},\dots,
x_{2^k}^{2},\dots,x_1^{n},\dots, x_{2^k}^{n}) \in {\mathbb{F}}_2^{n2^k,}
\]
by
%\begin{eqnarray*}  \label{eq:induce}
\[
\begin{aligned}\label{eq:induce2}
\xi_{S,S_1,S_2,t}({\mathbf{x}}) &= \xi_{S,S_1,S_2,t}((x_1^{1},\dots, x_{2^k}^{1},x_1^{2},\dots,
x_{2^k}^{2} \ldots,x_1^{n},\dots, x_{2^k}^{n})) \\
&= ( (\Sigma_{S,k}\circ\Gamma_{S_1,S_2,t})(\mathbf{x}^{1}),(\Sigma_{S,k}\circ\Gamma_{S_1,S_2,t})(\mathbf{x}^{2}),\dots,
(\Sigma_{S,k}\circ\Gamma_{S_1,S_2,t})(\mathbf{x}^{n}) )
\end{aligned}
\]
%\end{eqnarray*}
where $\mathbf{x}^j = (x_1^{j},\dots, x_{2^k}^{j}),$ for $1 \leq j \leq n.$ Since $T^{2^k}$ and $
\xi_{S,S_1,S_2,t}$ commute, $\Xi_{S,S_1,S_2,t}$ can be written as $T^{2^k} \circ \xi_{S,S_1,S_2,t}$ as well.
Now we are ready to provide the first characterization of $\theta$-cyclic codes over $B_k.$

\begin{lemma}[First characterization]
\label{lemchar}
Let $C$ be a code in $B_k^n$ and $\theta=\Theta_{S}\circ\Lambda_{S_1,S_2,t}$
be an automorphism in $B_k,$ for some $S,S_1,S_2\subseteq\{1,2,\dots,k\}$ and
an integer $t,$ where $0\leq t\leq r.$
Then, the code $\Phi_k(C)$ is fixed by the bijection $\Xi_{S,S_1,S_2,t} $
if and only if $C$ is a $\theta$-cyclic code.
\end{lemma}

\begin{proof}
Let $C$ be a $\theta$-cyclic code and let $ \mathbf{y}= (y_1,y_2,\dots,y_{n2^k}) \in \Phi_k(C).$  That is, there exists
$ \mathbf{x}=(x_1,x_2,\dots,x_{n})\in C $
such that     $ \y=  (\Phi_k(x_1),\Phi_k(x_2),\dots,\Phi_k(x_n))$
and  $\Phi_k(x_i)\in {\mathbb{F}}_{p^r}^{2^k}.$

We have
\[
\begin{aligned}
\Xi_{S,S_1,S_2,t}(\y) & =  \xi_{S,S_1,S_2,t} \circ T^{2^{k}}(  (\Phi_k(\x_1),\Phi_k(\x_2),\dots,\Phi_k(\x_n)))\\
      & =  \xi_{S,S_1,S_2,t} (  (\Phi_k(x_n),\Phi_k(x_1),\ldots,\Phi_k(x_{n-1})))\\
  & = ((\Sigma_{S,k}\circ\Gamma_{S_1,S_2,t})(\Phi_k(x_n)),  (\Sigma_{S,k}\circ\Gamma_{S_1,S_2,t})(\Phi_k(x_1)),\\
  & \quad \ldots, (\Sigma_{S,k}\circ\Gamma_{S_1,S_2,t})(\Phi_k(x_{n-1})))\\
  & = (\Phi_k(\theta (x_n)),  \Phi_k(\theta (x_1)),\dots,\Phi_k(\theta (x_{n-1})))\\
  & = \Phi_k(\theta (\x))\in \Phi_k(C),
\end{aligned}
\]
since $C$ is  a $\theta$-cyclic code.

Let   $C'$ be a code  over $\F_{p^r}^{n2^k}$
that is fixed  by  the  permutation $ \Xi _{S,S_1,S_2,t}$
and let $ \x= (x_1,x_2,\dots,x_{n}) \in \Phi_k^{-1} (C').$  Then, there exists
\[
\y= ( x_{1}^{1},\dots,
x_{1}^{2^k}, x_{2}^{1},\dots,x_{1}^{2^k},\dots, x_{n}^{1},\dots,
x_{n}^{2^k})\in C',
\]
such that
\[
\y= \Phi_{k}((x_1,x_2,\dots,x_{n})) =
(\Phi_k(x_1),\Phi_k(x_2),\dots,\Phi_k(x_n)).
\]
Since $C'$ is fixed by  the  permutation $ \Xixi $ and
$\Phi_k(x_i)\in {\mathbb{F}}_{p^r}^{2^k},$ we have that for $\Xixi(\y) \in C':$
\[
\begin{aligned}
\Xixi(\y) & =  \xixi \circ T^{2^{k}}((\Phi_k(x_1),\Phi_k(x_2),\ldots,\Phi_k(x_n)))\\
  &= \xixi ((\Phi_k(x_n),\Phi_k(x_1),\ldots,\Phi_k(x_{n-1})))\\
  &= ((\teta)(\Phi_k(x_n)),  (\teta)(\Phi_k(x_1)),\\
  &\quad \ldots,(\teta)\Phi_k(x_{n-1}))\\
  &= (\Phi_k(\theta(x_n)),\Phi_k(\theta(x_1)),\dots,\Phi_k(\theta(x_{n-1}))\\
  &= \Phi_k(\theta(x)).
\end{aligned}
\]
It follows that $\Phi_k(\theta(x)) \in C'$ and thus
  $\Phi_{k}^{-1}(C')$ is a $\theta$-cyclic code .
\end{proof}

Next we provide the second characterization of $\theta$-cyclic codes over $B_k.$ For this purpose,
let $\overline{\Psi}_k$ be the map extended from $\Psi_k,$ where for any $\mathbf{a}=(a_1,\dots,a_n)$ in $B_k^n,$ we have
\[
\overline{\Psi}_k(\mathbf{a})=(\Psi_k(a_1),\dots,\Psi_k(a_n)).
\]
Before providing the second characterization, we need the following lemma.

\begin{lemma}
Let $C$ be a subset of $B_k^n.$ Then, $C$ is a $B_k$-linear code with length $n$ if and only if there exist linear codes,
 $C_1,\dots,C_{2^k},$ over $\mathbb{F}_{p^r}$ such that
\[C=\overline{\Psi}_k^{-1}(C_1,\dots,C_{2^k}).\]
\label{linear}
\end{lemma}

\begin{proof}
$(\Longrightarrow)$ Given $C,$ there exist codes $C_1,\dots,C_{2^k}$ such that $C=\overline{\Psi}_k^{-1}(C_1,\dots,C_{2^k}),$
(as $\overline{\Psi}_k$ is a bijection).
For any $C_i,$ we want to show that it is really a linear code over $\mathbb{F}_{p^r}.$ For any $c_1,c_2\in C_i,$ consider
$c_j'=\overline{\Psi}_k^{-1}(0,\dots,0,c_j,0,\dots,0)\in C,$ for $j=1,2.$ Since $C$ is linear, we have
$\overline{\Psi}_k(c_1'+c_2')=(0,\dots,0,c_1+c_2,0,\dots,0)\in (C_1,\dots, C_s).$ Also, for any $\alpha\in \mathbb{F}_{p^r},$
there exists $\alpha'\in B_k$ such that
$\Psi_k(\alpha')=(0,\dots,\alpha,0,\dots,0)\in \mathbb{F}_{p^r}^{2^k},$ where $\alpha$ lies in the $i$-th coordinate. Then,
$\overline{\Psi}_k^{-1}(0,\dots,0,\alpha c_1,0,\dots,0)=\alpha'c_1'\in C.$ So, $\alpha c_1\in C_i$ for any $\alpha\in \mathbb{F}_{p^r}$
and $c_1\in C_i.$ Therefore,
$C_i$ is a linear code over $\mathbb{F}_{p^r}.$\\[0.25cm]
$(\Longleftarrow)$ For any $c_1,c_2\in C,$ since $C_i$ is linear for all $i,$
we have $\overline{\Psi}_k(c_1+c_2)\in (C_0,\dots,c_s).$ So, $c_1+c_2\in C.$ Also, for
any $\beta\in B_k,$
we have $\overline{\Psi}_k(\beta c_1)=\Psi_k(\beta)\overline{\Psi}_k(c_1)\in (C_1,\dots,C_s),$
which means $\beta c_1\in C$
for any $\beta\in B_k$ and $c_1\in C,$ as we hope.
\end{proof}

Let $\tilam$ be a permutation on $\{1,2,\dots,2^k\}$ induced by
$\Theta_S\circ\Lambda_{S_1,S_2,t},$ and $\ord(\tilam)$ be the order of $\tilam.$
The following theorem also gives a
characterization for $\theta$-cyclic codes over the ring $B_k.$

\begin{theorem}[Second characterization]
A linear code $C$ over $B_k$ is $\theta$-cyclic of length $n$ if and only if
there exist quasi-$\tilde{\theta}$-cyclic codes $C_1,C_2,\dots,C_{2^k}$
of length $n$ over $\mathbb{F}_{p^r}$
with index $\ord(\tilam),$ such that
\[
C=\overline{\Psi}_k^{-1}(C_1,C_2,\dots, C_{2^k})
\]
where $\tilde{\theta}=\phi^{t\ord(\tilam) },$ for some $t$ as in Lemma \ref{autodecomp},
with $\phi$ is the Frobenius automorphism in
$\mathbb{F}_{p^r},$ and
$T_{\tilde{\theta}}(C_i)\subseteq C_j,$ where $j\in S\cup S_2,$ for all
$i=1,2,\dots,2^k.$
\label{charthetacyclic}
\end{theorem}
\begin{proof}
$(\Longrightarrow)$ By Proposition \ref{linear},
we can find codes over $\mathbb{F}_{p^r},$ $C_1,C_2,\dots,C_{2^k},$ such that,
\[C=\overline{\Psi}_k^{-1}(C_1,C_2,\dots,C_{2^k}).\]
For any $c_i\in C_i,$ let $c_i=(\alpha_1,\dots, \alpha_{n}).$ If,
$c=\overline{\Psi}_k^{-1}(0,\dots,0,c_i,0,\dots,0),$ then
\[
\left(\alpha_1v_{H_i}-\sum_{H\in\mathcal{H},~H\supsetneq H_i}\alpha_1v_H,\dots,
\alpha_nv_{H_i}-\sum_{H\in\mathcal{H},~H\supsetneq H_i}\alpha_nv_H\right).
\]
So, if we consider
\[
\overline{\Psi}_k(T_{\theta}^{t_1}(c))=(0,\dots,0,T_{\tilde{\theta}}^{t_1}(c_i),0,\dots,0),
\]
then we have
$T_{\tilde{\theta}}(c_i)$ is in $C_j,$ where $j\in S\cup S_2.$
By continuing this process, we have $T^{\ord(\tilam)}_{\tilde{\theta}}(c_i)\in C_i,$
which means, $C_i$ is quasi-$\tilde{\theta}$-cyclic code over $\mathbb{F}_{p^r}$
with index $\ord(\tilam),$ for all
$i=1,\dots,2^k.$\\

$(\Longleftarrow)$ For any $c\in C,$ we can see that
$\overline{\Psi}_k(c)\in (C_1,\dots,C_{2^k}).$
Since $C_i$ is quasi-$\tilde{\theta}$-cyclic
code over $\mathbb{F}_{p^r}$
with index $\ord(\tilam),$ for all
$i=1,\dots,2^k,$ $C_1,$
%and $C_{2^k}$ are $\tilde{\theta}$-cyclic codes,
and $T_{\tilde{\theta}}^{t_1}(C_i)\subseteq C_j,$ where $j\in S\cup S_2,$ for all
$i=1,2,\dots,2^k,$ where $1\leq t_1\leq 2^k.$
Then we have  $T_{\theta}(c)=\overline{\Psi}_k^{-1}(T_{\tilde{\theta}}(\Psi_k(c)))\in C,$ as we hope.
\end{proof}

Theorem~\ref{charthetacyclic} gives us an algorithm to construct skew-cyclic codes over the ring $B_k$ as follows.
\vspace{0.6cm}
\begin{alg}
Given $n,$ the ring $B_k,$ and an automorphism $\theta.$

\begin{itemize}
 \item[(1)] Decompose $\theta$ into $\theta=\Theta_{S}\circ\Lambda_{S_1,S_2,t}.$

\item[(2)] Determine $\ord(\tilam)$ and $\tilde{\theta}=\theta|_{\mathbb{F}_{p^r}}=\phi^t,$ where
	$\phi$ is the Frobenius automorphism in $\mathbb{F}_{p^r}.$

\item[(3)] Choose quasi-$\tilde{\theta}$-cyclic codes over $\mathbb{F}_{p^r},$ say $C_1,\dots,C_{2^k},$ such that
\[
T_{\tilde{\theta}}^{t_1}(C_i)\subseteq C_{j},
\]where $j\in S\cup S_2,$ for all $i=1,2,\dots,2^k.$

\item[(4)] Calculate $C=\overline{\Psi}_k^{-1}(C_1,\dots,C_{2^k}).$

\item[(5)] $C$ is a $\theta$-cyclic code over the ring $B_k.$
\end{itemize}
\end{alg}

\section{Skew-polynomial rings, self-dual codes, and optimal codes}

At the end of this section, we construct some optimal self-dual $\theta$-cyclic codes.
For the sake of construction, we provide a connection between $\theta$-cyclic codes
and  polynomial rings over $B_k.$ Also, we characterize self-dual $\theta$-cyclic codes.

Let $\theta$ be an automorphism in $B_k.$ We define a polynomial ring $B_k[x;\theta]$
with usual addition, and its multiplication defined as
\[
ax^i*bx^j=a\theta^i(b)x^{i+j}.
\]
We can see from the above multiplication that $B_k[x;\theta]$ is not a commutative ring in general.
We call this ring \emph{skew polynomial ring}.

Let $R=B_k[x;\theta]/(x^n-1)$ and define a left action of $B_k[x;\theta]$ on $R$ by
\[
h(x)*a:=h(x)*f(x)+(x^n-1),
\]
for $a=f(x)+(x^n-1) \in R$ and $h(x)\in B_k[x;\theta].$  It is easy to see that this
left action is well-defined and $R$ is a left module over $B_k[x;\theta].$

Every elements in $B_k^n$ could be associated to a polynomial over $B_k$ by the following map,
\[
\begin{array}{llll}
   \tau : & B_k^n & \longrightarrow & R \\
	     & c=(c_0,c_1,\dots,c_{n-1}) & \longmapsto & \sum_{i=0}^{n-1} c_ix^i
\end{array}
\]
Moreover, $\tau$ is a bijection between $B_k^n$ and $R.$ Then, we have the following result
({\it c.f.} Theorem 17 in \cite{irw}).

\begin{theorem}
A linear code $C$ over $B_k$ is $\theta$-cyclic if and only if $\tau(C)$ is a left $B_k[x;\theta]$-submodule
of $B_k[x;\theta]/(x^n-1).$
\end{theorem}

%\begin{proof}
%The proof is similar to the proof of Theorem 17 in \cite{irw}.
%\end{proof}

For generators of $\theta$-cyclic codes, we have the following result.
\begin{proposition}
Let $C=\overline{\Psi}_k^{-1}(C_1,\dots,C_{2^k})$ be $\theta$-cyclic codes over $B_k,$
where $C_1,\dots,C_{2^k}$ are codes over $\mathbb{F}_{p^r}.$
If $C_i=\langle g_{1_i}(x),\dots,g_{m_i}(x)\rangle,$ for all $i=1,\dots,2^k,$ then
\[
C=\langle g_{1_1}(x),\dots,g_{m_1}(x),\dots,g_{1_s}(x),\dots,g_{m_s}(x)\rangle
\]in some sense.
\label{gen1}
\end{proposition}

\begin{proof}
For any $c(x)\in C,$ there exist $c_i(x)\in C_i,$
where $1\leq i\leq 2^k,$ such that $c(x)=\overline{\Psi}_k^{-1}(c_1(x),\dots,c_{2^k}(x)).$ Now, let
\[
c_j(x)=\sum_{k=1}^{m_j}\alpha_{kj}(x)g_{jk}
\]
for all $j=1,\dots,2^k,$ then we have
\[
\begin{aligned}
c(x) & =  v_{H_1}\left(\sum_{k=1}^{m_1}\alpha_{k1}(x)g_{1k}\right)\\
     &   +\dots+v_{H_i}\left(\sum_{k=1}^{m_i}\alpha_{ki}(x)g_{ik} -\sum_{H_j\subseteq H_i}
     \left(\sum_{k=1}^{m_j}\alpha_{kj}(x)g_{jk}\right)\right) \\
     &   +\dots+v_{H_{2^k}}\left(\sum_{k=1}^{m_{2^k}}\alpha_{k2^k}(x)g_{2^k k}-\sum_{j=1}^s
     \left(\sum_{k=1}^{m_j}\alpha_{kj}(x)g_{jk}
	  \right)\right)
\end{aligned}
\]
as we hope.
\end{proof}

For $c_1=(c_{10},\dots,c_{1n-1}),c_2=(c_{20},\dots,c_{2n-1})\in B_k^n,$ define
\[
\langle c_1,c_2\rangle=\sum_{i=0}^{n-1}c_{1i}c_{2i}
\]
which we call {\it Euclidean product}. For a  code $C\subseteq B_k^n,$ we define
$C^\bot =\{b\in B_k^n|\langle b,c\rangle=0,\forall c\in C\}.$ If $C\subseteq C^\bot,$ then we call $C$
{\it self-orthogonal}, and if $C=C^\bot,$ then $C$ called {\it Euclidean self-dual}.
For a linear Euclidean self-dual code over $B_k,$ we have the characterization below
(follows from \cite[Theorem 6.4]{Dough3}).

\begin{proposition}
Let $C=\overline{\Psi}_k^{-1}(C_1,\dots,C_{2^k}).$ $C$ is a self-dual code over $B_k$
if and only if $C_1,\dots,C_{2^k}$ are also self-dual codes
over $\mathbb{F}_{p^r}$
\label{chareuc}
\end{proposition}

%\begin{proof}
%This follows from \cite[Theorem 6.4]{Dough3}.
%\end{proof}

As an immediate consequence, by applying Theorem \ref{charthetacyclic} and Proposition \ref{chareuc},
 we have the following result.

\begin{theorem}
A linear code $C$ over $B_k$ is Euclidean self-dual $\theta$-cyclic of length $n$
if and only if there exist self-dual quasi-$\tilde{\theta}$-cyclic
codes $C_1,C_2,\dots,C_{2^k}$ of length $n$ such that
\[
C=\overline{\Psi}_k^{-1}(C_1,C_2,\dots, C_{2^k})
\]
where $\tilde{\theta}=\phi^{t},$ for some $t$ as in Lemma \ref{autodecomp}, with $\phi$ is the Frobenius automorphism in
$\mathbb{F}_{p^r},$ and
$T_{\tilde{\theta}}^{t_1}(C_i)\subseteq C_{\tilam^{t_1}(i)}$ for all
$i=1,2,\dots,2^k,$ where $1\leq t_1\leq 2^k.$
\end{theorem}

%\begin{proof}
%Apply Theorem \ref{charthetacyclic} and Proposition \ref{chareuc}.
%\end{proof}

The following lemma gives the minimum distance for codes over $B_k.$
\begin{lemma}
Let $C$ be a linear code over $B_k.$ If $C=\overline{\Psi}_k^{-1}(C_1,\dots,C_{2^k}),$
for some codes $C_1,\dots,C_{2^k}$  over $\mathbb{F}_{p^r},$ then
$d_H(C)=\min_{1\leq i\leq 2^k} d_H(C_i).$
\label{jarak}
\end{lemma}

\begin{proof}
Suppose $\min_{1\leq i\leq 2^k} d_H(C_i)=d_H(C_j)$ for some $j,$ and let $c\in C_j$
such that $d_H(c)=d_H(C_j),$ then we have
\[
d_H(\overline{\Psi}_k^{-1}(0,\dots,0,c,0,\dots,0))=d_H(C_j).
\]
 So, $d_H(C)=\min_{1\leq i\leq 2^k} d_H(C_i).$
\end{proof}

As an immediate consequence of Lemma \ref{jarak}, we have the following result.

\begin{corollary}
 If $C=\overline{\Psi}_k^{-1}(C_1,\dots,C_{2^k})$ be a $\theta$-cyclic code over $B_k,$ where $C_i$ are
quasi-$\tilde{\theta}$-cyclic code of index $\ord(\tilam)$ over $\mathbb{F}_{p^r}$ as in
Theorem \ref{charthetacyclic}, for all $i,$ then
\[
d(C)=\min\{d(C_1),\dots,d(C_{2^k})\}.
\]
\end{corollary}

This means, we can use the optimal $\theta$-cyclic codes or
quasi-$\theta$-cyclic codes over $\mathbb{F}_{p^r}$ to construct optimal $\theta$-cyclic
codes over $B_k$ for all $k$ with respect to Hamming weight.
In \cite{boucher}, Boucher and Ulmer find some optimal
$\theta$-cyclic codes over finite fields, including Euclidean self-dual code $[36,18,11]$
which improves previous known bound for self-dual codes of
length $36.$ We use this optimal
codes to construct optimal $\theta$-cyclic codes over $B_k.$
Some example of optimal Euclidean self-dual codes are shown in Table 1.
Here we only use $p=2,$ $r=2,$
$\theta|_{\mathbb{F}_{2^2}}$ is a Frobenius automorphism, and $\alpha$
is a generator of $\mathbb{F}_{2^2}^\times.$
Also, the given generators generate $\theta$-cyclic codes in the way as in the
proof of Proposition \ref{gen1}.

\begin{center}
\begin{table}[b!]
\caption{Table of examples of optimal $\theta$-cyclic codes}
\begin{tabular}{|c|c|c|c|c|}\hline\hline
$n$ & $d$ & Generator polynomial & $B_k$ & $\theta$\\[2mm]\hline\hline
 4 & 3 & $x^2+\alpha^2 x+\alpha$ & $B_2$ & \begin{tabular}{l}
            $v_2\mapsto v_1$\\
	    $v_1\mapsto v_2$\\
           \end{tabular}\\ \hline
12 & 6 & $x^{6}+x^{5}+\alpha^2x^4+x^3+\alpha x^2+x+1$ & $B_3$ & \begin{tabular}{l}
            $v_1\mapsto v_2$\\
	    $v_2\mapsto v_3$\\
	    $v_3\mapsto v_1$\\
\end{tabular} \\\hline
 20 & 8 &\begin{tabular}{l}
          $x^{10}+\alpha^2x^{9}+\alpha x^8+x^7+x^6$\\
	  $ +x^4+x^3+\alpha x^2+\alpha x+1$
         \end{tabular}
 & $B_4$ & \begin{tabular}{l}
            $v_1\mapsto v_2$\\
	    $v_2\mapsto v_1$\\
	    $v_3\mapsto 1-v_4$\\
	    $1-v_4\mapsto v_3$\\
          \end{tabular} \\ \hline
36 & 11 & \begin{tabular}{l}
          $x^{18}+x^{16}+\alpha^2 x^{15}+\alpha x^{14}+\alpha^2 x^{13}+x^{12}$\\
	  $ +\alpha x^{10}+\alpha x^9+\alpha x^8+\alpha^2 x^6+x^5+\alpha x^4$\\
	  $ +x^3+\alpha^2 x^2+\alpha^2$
         \end{tabular} & $B_6$ & \begin{tabular}{l}
            $v_1\mapsto 1-v_2$\\
	    $1-v_2\mapsto v_3$\\
	    $v_3\mapsto v_1$\\
	    $v_4\mapsto v_5$\\
	    $v_5\mapsto v_6$\\
	    $v_6\mapsto v_5$\\
          \end{tabular} \\ \hline
40 & 12 & \begin{tabular}{l}
          $x^{20}+x^{17}+\alpha^2 x^{15}+\alpha x^{14}+\alpha^2 x^{13}+\alpha^2 x^{12}$\\
	  $ +x^{11}+ x^9+\alpha x^8+\alpha x^7+\alpha^2 x^6+\alpha x^5+x^3+1$
         \end{tabular} & $B_7$ & \begin{tabular}{l}
            $v_1\mapsto 1-v_2$\\
	    $1-v_2\mapsto v_3$\\
	    $v_3\mapsto v_1$\\
	    $v_4\mapsto v_5$\\
	    $v_5\mapsto v_6$\\
	    $v_6\mapsto v_5$\\
	    $v_7\mapsto v_7$\\
          \end{tabular} \\

\hline
\end{tabular}\\
\end{table}
\end{center}

\section*{Acknowledgement}
D.S. and I are supported by \emph{Riset ITB 2016} and \emph{Penelitian Unggulan Perguruan Tinggi (ITB-Dikti) 2017.}


\begin{thebibliography}{99}
\bibitem{abualrub}
T. Abualrub, N. Aydin, and P. Seneviratne, "On $\theta$-Cyclic Codes over $\mathbb{F}_2+v\mathbb{F}_2,"$
{\it Australasian Journal of Combinatorics}
vol. 54, 2012, 115-126.


\bibitem{boucher}
D. Boucher, and F. Ulmer, "Coding with Skew Polynomial Rings," \emph{Journal of Symbolic Computation} No. 44, 2009, 1644-1656.

\bibitem{bou-geis}
D. Boucher, W. Geiselmann, and F. Ulmer, Skew Cyclic Codes, \emph{Appl. Algebra Eng. Commun. Comput.} 18, 2007, 379-389.

\bibitem{dougherty-ceng}
Y. Cengellenmis, A. Dertli, and S. Dougherty, Codes over an Infinite Family of Rings with a Gray Map,
\emph{Designs, Codes, and  Cryptography}, 72(3), 2014, 559-580.

\bibitem{Dough3} S.T. Dougherty, J.L. Kim, and H. Kulosman, "MDS Codes over Finite Principal Ideal Rings,"
\emph{Design, Codes, and Cryptography,}  50(1), 2009, 77-92.

\bibitem{gao}
J. Gao, "Skew Cyclic Codes over $\mathbb{F}_p+v\mathbb{F}_p,$" {\it J. Appl. Math. and Informatics} 31(3-4), 2013, 337-342.

\bibitem{irw}
Irwansyah, A. Barra, S.T. Dougherty, A. Muchlis, I. Muchtadi-Alamsyah, I., P. Sol\'{e}, D. Suprijanto, D., and O. Yemen,
"$\Theta_S$-Cyclic Codes over $A_k,$" \emph{International Journal of Computer Mathematics: Computer System Theory},
1(1), 2016, 14-31.
\end{thebibliography}
\end{document}